\documentclass[11pt]{article}
\usepackage{amsmath,amsthm,amssymb,latexsym,graphicx,amsfonts,amscd}

\usepackage{color}

\newtheorem{theorem}{Theorem}[section]

\newtheorem{proposition}[theorem]{Proposition}
\newtheorem{lemma}[theorem]{Lemma}
\newtheorem{corollary}[theorem]{Corollary}

\newtheorem{example}[theorem]{Example}

\numberwithin{equation}{section}
\def\Pf{\smallskip\goodbreak{\sl Proof. }}

\def\Fp{\vadjust{}\penalty200 \hfill
\lower.3333ex\hbox{\vbox{\hrule\hbox{\vrule\phantom{\vrule height
6.83333pt depth 1.94444pt width 8.77777pt}\vrule}\hrule}}
\ifmmode\let\next\relax\else\let\next\par\fi \next}

\def\End{\mathop{\rm End}\nolimits}

\def\id{\mathop{\rm id}\nolimits}

\def\Hom{\mathop{\rm Hom}\nolimits}

\def\cf{\mathop{\rm cf}\nolimits}

\def\Ker{\mathop{\rm Ker}\nolimits}
\def\Coker{\mathop{\rm Coker}\nolimits}

\def\Br{\mathop{\rm Br}\nolimits}
\def\dom{\mathop{\rm Dom}\nolimits}

\def\rk{\mathop{\rm rk}}

\def\Z{{\mathbb Z}}

\def\bS{{\mathbb S}}

\def\a{\alpha}

\def\k{\kappa}
\def\hv{\widehat\va}
\def\va{\varphi}

\def\l{\lambda}

\def\id{{\rm id\,}}

\def\sup{{\rm sup\,}}

\def\Im{{\rm Im\,}}

\def\Dom{{\rm Dom}\,}

\def\aln{{\aleph_0}}
\def\Cont{2^{\aln}}

\def\n+1d{{}^{ n+1 \downarrow }\l}

\def\size#1{\left|\,#1\,\right|}

\def\Rhat{\widehat R}
\def\Khat{\widehat K}
\def\Fhat{\widehat F}

\def\Chat{\widehat C}
\def\Hhat{\widehat H}
\def\Ehat{\widehat E    }

\def\hv{\widehat{\varphi}}
\def\hp{\widehat{\psi}}
\def\hb{\widehat b}
\def\hH{\widehat H}
\def\hB{\widehat B}
\def\hE{\widehat E}

\def\restr{\mathop{\upharpoonright}}
\def\to{\rightarrow}
\def\arr{\longrightarrow}

\def\iff{\Longleftrightarrow}

\def\Rhat{\widehat{R}}

\begin{document}

\title{\bf Cellular covers of cotorsion-free modules}
\footnotetext{The first and third authors were supported by the
project No. 963-98.6/2007 of the German-Israeli
Foundation for Scientific Research \& Development and the project No. AOBJ 548025 of the German Research Foundation.

The second author was supported by the Spanish Ministry of
Education and Science MEC-FEDER grant MTM2007-63277.

}

\author{ R\"udiger G\"obel, Jos\'e L. Rodr\'{\i}guez, Lutz Str\"ungmann}
\date{September 16th, 2009}
\maketitle

\begin{abstract}
In this paper we improve recent results dealing with cellular
covers of $R$-modules. Cellular covers (sometimes called co-localizations) 
come up in the context of homotopical localization of topological
spaces. They are related to idempotent cotriples, idempotent comonads or
coreflectors in category theory.

Recall that a homomorphism of $R$-modules $\pi: G\to H$ is called
a {\it cellular cover} over $H$ if $\pi$ induces an isomorphism
$\pi_*: \Hom_R(G,G)\cong \Hom_R(G,H),$ where $\pi_*(\varphi)=
\pi \varphi$ for each $\varphi \in \Hom_R(G,G)$ (where maps are
acting on the left). On the one hand, we show that every
cotorsion-free $R$-module of rank $\kappa<\Cont$ is realizable as
the kernel of some cellular cover $G\to H$ where the rank of $G$ is $3\kappa +1$ (or 3, if
$\kappa=1$). The proof is based on Corner's classical idea of how
to construct torsion-free abelian groups with prescribed countable
endomorphism rings. This complements results by Buckner--Dugas
\cite{BD}. On the other hand, we prove that every cotorsion-free
$R$-module $H$ that satisfies some rigid conditions admits arbitrarily large
cellular covers $G\to H$. This improves results by Fuchs--G\"obel \cite{FG} and
Farjoun--G\"obel--Segev--Shelah \cite{FGSS07}.

\end{abstract}

\section{Introduction}\label{introduction}
Cellular covers of groups and modules are the
algebraic analogues of the cellular
approximations of topological spaces due to J. H. C.
Whitehead. These feed into the context
of homotopical localization in closed model categories
established by Bousfield, Farjoun, Hirschhorn, and others
(see e.g. \cite{Bo77}, \cite{Bo97},
\cite{Cha96}, \cite{F97}, \cite{Hir02}, \cite{Nof99}).
In some special cases there is even a good interplay between
cellularization of spaces and cellularization of groups
via the fundamental group \cite{RS01}, as was previously obtained
for localizations in \cite{Cas95}, \cite{Bo97}, \cite{Cas00}, \cite{CRT98}.
For instance, the universal central extension 
$0\to H_2(H;\Z)\to  \widetilde H\to H \to 1$
of a perfect group $H$ yields a surjective cellular cover, with kernel 
the Schur multiplier $H_2(H;\Z)$.  This central extension is the one
induced on the lowest homotopy groups of the 
fiber sequence $AX\to X\to X^+$, where $X\to X^+$ is the Quillen plus-construction, 
$AX\to X$ is the acyclic cellular approximation, and $X=K(H,1)$ is the Eilenberg--Mac Lane
space with fundamental group $H$; see \cite{RS01}.
Other motivating examples can be found in
\cite{RS01}, \cite{MP01}, \cite{FGS07}, \cite{Flo}, \cite{RS07}.

Recall that a homomorphism $\pi:G \to H$ of groups
is a {\it cellular cover} over $H$,
if every homomorphism $\varphi: G\to H$ lifts uniquely
to an endomorphism $\widetilde\varphi$ of $G$
such that $\pi \widetilde\varphi = \varphi$.
In such case $\pi: G\to \Im(\pi)$ is a cellular cover over $\Im(\pi)$, 
hence one can assume without loss of generality that
$\pi: G\to H$ is an epimorphism. We then say that
\begin{equation}
\label{celular-sequence}
0\to K \to G\overset{\pi}{\to} H\to 1
\end{equation}
is a {\it cellular exact sequence}.

One of the main objectives is to
classify (up to isomorphism) all possible cellular exact sequences
with either fixing the cokernel $H$ or the kernel $K$.
It is then crucial to know whether there is a set or a proper class
(up to isomorphism) of cellular exact sequences (\ref{celular-sequence})
for a fixed $K$ or $H$. Certainly, it is more desirable
to find cellular covers of any given cardinal $\lambda\geq |K|$ or $|H|$.
Here we are guided by similar results obtained for localizations;
see e.g. \cite{Cas00}, \cite{D04}, \cite{D07},
\cite{DMV87}, \cite{GRS02}, \cite{GS02}, \cite{Li00}.

First observe that $K$ must be central in $G$,
and conversely every abelian group $K$ is the Schur
multiplier $H_2(H;\Z)$ of some perfect group $H$,
thus all abelian groups can appear
as kernels of cellular covers \cite{FGS07}.
It has been also proved in \cite{FGS07} that
$G$ is abelian, nilpotent, or an $R$-module, whenever $H$ is abelian,
nilpotent or an $R$-module, respectively, where $R$ is any commutative ring with unit 
(compare with the case of localizations \cite{Li00}, \cite{CRT98}).
Of course, other properties like for example being perfect  (see \cite{Se08})
are not transferred in general (cf. \cite{RSV05}). 
Further considerations on cellular covers of arbitrary
groups have been recently achieved in \cite{CDFS08} and 
\cite{FGSS07}.


Recall some known results for cellular covers of abelian groups.
If $H$ is divisible, then $G$ in (\ref{celular-sequence})
 can be determined explicitly as shown in
\cite{CFGS07}. A different proof of this result using Maltis
duality theory is given in \cite{FG}. If $H$ is reduced, then $K$
must be cotorsion-free, see \cite{BD}, \cite{FGSS07}, \cite{FG}.
And if $H$ is torsion and reduced, then the cellular exact sequence
collapses and $K=0$, see \cite{FG}. Furthermore, if $K$ is free, then it is
very easy to see (\cite{FGSS07}) that $|K|\leq |H|$. 

The situation becomes more exciting for cotorsion-free abelian groups.
Buckner and Dugas showed in \cite{BD} that if K is cotorsion-free, 
then there exist arbitrarily large cellular covers $G\arr H$ with kernel $K$. 
Hence (\ref{celular-sequence}) runs over a proper class in that case.
By their construction it follows $|G|\ge 2^{\aleph_0}$.
Here \cite{BD} uses a construction based on the combinatorial
principle Strong Black Box from \cite{GW} which we will replace by
the ordinary Black Box, thus filling in missing cardinals
$\k^{\aln}$ for the size of the kernels, see Corollary
\ref{bd-paper}. However, due to the nature of the Black Box this
does not say anything about cellular covers of size below the
continuum - one problem that we want to attack in the present
paper. We would like to point out that these Black Boxes are theorems in ZFC,
thus do not depend on additional axioms of set theory; see for instance \cite{GT}.

Dually, for every infinite cardinal $\lambda$ there exists a
cotorsion-free abelian group $H$ of cardinal $\lambda$, which
admits arbitrary large cellular covers $G$'s (see \cite{FGSS07}).
The proof is based on \cite{GM90} concerning the existence of
arbitrarily large indecomposable vector spaces with four
distinguished subspaces. For instance every rank one group that is
not a ring has arbitrarily large cellular covers (see \cite{FG}).
Note that this result does not fix the group $H$, but the cardinal
$\lambda$.

In the present paper we show the following new contributions to
the theory of cellular covers: As indicated above we consider the
existence of cellular covers of size below the continuum. 
In this case cotorsion-free is the same as torsion-free and reduced; 
see \cite{GT}. And if  $K$
is cotorsion-free of rank $\k <\Cont$, then we prove that there
exists a cellular cover $G$ of rank $3$, $3\k +1$, or $\k$,
respectively if $\k$ is 1, finite and greater or equal than 2, or
infinite; see Theorem~\ref{crucial-kernels}. This explains our
interest in extending the main result of \cite{BD} to Corollary
\ref{bd-paper} mentioned above. Dually, if $H$ is cotorsion-free
of size $\k<\Cont$ and $\End(H)=\Z$, then
there exists a cellular cover $G$ of size $\k$, see
Theorem~\ref{crucial-quotients}. Looking at cokernels of size larger than
the continuum  we are also able to find arbitrarily large cellular
sequences with prescribed cokernel. This is our main result
(Theorem \ref{main2}): If $H$ is
cotorsion-free of size $\k\geq \Cont$ and satisfies
$\End(H)=\Z$ and $\Hom(H,M)=0$ for all
$\aleph_0$-free abelian groups $M$,
then there exist arbitrarily large cellular covers $G$.
To get this result we have to modify the classical Black Box
to be suitable for this purpose. 

It is needless to say that our results hold for $R$-modules, and
are stated in broader generality as indicated here (see Section 2).

We finally remark that cellular covers of groups provide
(singly cogenerated) colocalization functors in the category of groups 
as noticed in \cite{FGS07}.
In particular, they can be translated to spaces
by simply taking Eilenberg--Mac Lane spaces
as in \cite{RS01} or \cite{GRS02}. That is, if $G\to H$ is a surjective cellular
cover of groups then $K(G,n) \to K(H,n)$ is a cellular approximation of spaces
(assumed $H$ abelian for $n\geq 2$). 

\bigskip
\noindent
{\it Acknowledgments:} The second author would like to thank the University of
Duisburg-Essen for its hospitality during his visit in
Summer 2008.

\section{Cellular covers of modules}

A homomorphism of $R$-modules $\pi: G\to H$ is called  a {\it
cellular cover} over $H$ if $\pi$ induces an isomorphism
$$\pi_*: \Hom_R(G,G)\cong \Hom_R(G,H),$$
where $\pi_*(\varphi)= \pi \varphi$ for each $\varphi \in
\Hom_R(G,G)$ (where maps are acting on the left). For $R=\Z$ these
are precisely cellular covers (or co-localizations) of abelian
groups (see e.g. \cite{FGS07}, \cite{FG}).
Recall that $\pi: G\to H$ is a {\it localization} if it induces an isomorphism
 $\pi^*: \Hom_R(H,H) \cong \Hom_R(G,H)$, by $\pi^*(\varphi)= \varphi \pi$.

If $\pi: G\to H$ is a cellular cover of $R$-modules, then $\pi$
induces a morphism $\End_{R}(H)\to \End_{R}(G)$, given by
$\varphi\mapsto \tilde\varphi$, where $\tilde\varphi:G\to G$ is the
unique lifting of $\varphi$, i.e. such that $\pi \tilde \varphi =
\varphi \pi$. In fact, it is a homomorphism of $R$-algebras, also by
uniqueness of liftings. The first part of the following
result can be found in \cite{FG}.

\begin{proposition}
\label{endo} Let $\eta: R_0\to R$ be a homomorphism of rings, $H$
be an $R$-module and $\pi: G\to H$ a cellular cover as
$R_0$-modules. Then $G$ admits a unique $R$-module structure for
which  $\pi: G\to H$ is a morphism of $R$-modules. Furthermore,
$\pi$ is also a cellular cover viewed as $R$-modules, if $\eta (R_0)$ is
central in $R$.
\end{proposition}
\begin{proof}
The last observation holds as follows. If $\varphi: G\to H$ is an
$R$-homomorphism, then it is an $R_0$-homomorphism (via $\pi$),
hence it lifts to a unique $R_0$-homomorphism $\widetilde \varphi:
G\to G$ such that $\pi \widetilde \varphi = \varphi$. For a fixed
$r\in R$ we have that left multiplication by $r$ on $G$ is a
$R_0$-homomorphism since $\eta (R_0)$ is central in $R$.
Therefore, $r \widetilde \varphi$ and $\widetilde \varphi r$ are
two $R_0$-homomorphisms such that $\pi r \widetilde \varphi = \pi
\widetilde \varphi r$ both equal $r \pi$, since $\pi$ is a
$R$-homomorphism. Because $\pi$ is a $R_0$-cellular cover it
follows that $r \widetilde \varphi=\widetilde \varphi r$, and
hence $\widetilde \varphi$ is an $R$-homomorphism.
\end{proof}
In particular, for $R_0=\Z$ and $R$ a commutative ring with unit,
we have that cellular  covers $G\to H$ as
abelian groups over an $R$-module $H$ are also cellular covers as
$R$-modules (see  Proposition 2.6 in \cite{FG}).

The following easy observation allows us to 
consider surjective cellular covers:
\begin{proposition}
\label{surjective-cover}
A homomorphism of $R$-modules $G\to H$ is a cellular cover if
and only if $\pi: G\to \Im (\pi)$ is a cellular cover 
and $\Hom(G,\Coker \pi)=0$. \qed
\end{proposition}

In this paper we will construct surjective cellular covers $G\to H$
which are also localizations (cf. \cite{FGS07}).
These properties are easy to verify when the involved modules
are {\it rigid} in the sense that $\End_R(G)=R = \End_R(H)$.
The following fact is immediate and will be used in our theorems:

\begin{proposition}\label{endo1} Let $\pi: G\to H$ be an epimorphism of
$R$-modules such that $H$ is $R$-torsion-free and suppose that $G$
is rigid. Then $\pi$ is a cellular cover if and only if
$\Hom_R(G,H)=\pi R$.
In that case $H$ is rigid as well, and $\pi$ is a localization.
\qed
\end{proposition}
\begin{proof}
If $\End_R(G)=R$ and $H$ is $R$-torsion-free then $\Hom_R(G,\Ker \pi)=0$ and therefore $\pi_*$
is injective. It is also clear that $\pi_*$ is surjective if and
only if $\Hom_R(G,H)= R \pi $. The last statement is immediate.
\end{proof}

We now fix some notation and setting from \cite{GT} about cotorsion-free modules,
from which we will build up our desired rigid modules.

Let $R$ be a commutative
ring with $1$ and a distinguished countable multiplicatively
closed subset $\bS=\{ s_n : n \in \omega \}$ such that $R$ is
$\bS$-reduced and $\bS$-torsion-free. Thus $\bS$ induces a Hausdorff topology on $R$,
taking $q_mR$ ($m \in \Z)$ as the neighborhoods of zero where
$q_m=\prod_{n < m}s_n$. We let $\Rhat$ be the $\bS$-adic
completion of $R$. We will also assume that $R$ is cotorsion-free
(with respect to $\bS$), this is to say that $\Hom(\Rhat,R) =0$.
More generally, an $R$-module $M$ is
{\it $\bS$-cotorsion-free} if $\Hom_R(\Rhat, M)=0$.
We must say what it means that $M$ has
rank $\k \le\size{R}$. (Note that $R$ may not be a domain.) If
$\size{M} > \size{R}$ it suffices to let $\rk(M)=\size{M}$. If
$\size{M}\le \size{R}$, then $\rk(M)= \k$ means that there is a
free submodule $E =\bigoplus_{i < \k}Re_i$ of $M$ such that $M/E$
is $\bS$-torsion. (Note that $E$ also exists if $\size{M} >
\size{R}$.) Recall that $M$ is $\bS$-torsion if for all $m\in M$
there is $s\in \bS$ such that $sm=0$. Similarly, a submodule $N$
of $M$ is $\bS$-pure if $sM \cap N=sN$ for all $s \in \bS$. If $M$
is $\bS$-torsion-free and $N \subseteq M$, then we denote by $N_*$
the smallest pure submodule of $M$ containing $N$, i.e. $N_*=\{m
\in M | \exists s \in \bS \textrm{ and } sm \in N \}$.

We will write $\Hom(M,N)$ for $\Hom_R(M,N)$ and in what follows all
appearances of torsion, pure, etc. refer to $\bS$ and we will
therefore not mention the underlying set $\bS$.


\section{A theorem about kernels of cellular covers}\label{kernl}

Using the Strong Black Box from \cite{GW} it was shown in
\cite{BD} that any cotorsion-free $R$-module $K$ can be the
kernel of a cellular cover of arbitrarily large cardinality
$\kappa$. Analyzing the proof by Buckner and Dugas \cite{BD} we
note that easily the Strong Black Box can be replaced by the
(general) Black Box as in \cite{CG} or \cite{GT}. Thus the
existence of cotorsion-free kernels of cellular covers extends to
a wider spectrum of cardinals. We skip the proof of the following
corollary concerning uncountable cardinals $\k$ since the
necessary changes can be deduced from our proof of Theorem
\ref{main2}.

\begin{corollary} \label{bd-paper} Let $R$ be a commutative, cotorsion-free  ring with $1$
and $\k$ be any infinite cardinal with  $\k^{\aln}>\size{R}$. If $K$
is a cotorsion-free $R$-module of size $\k$, then there is a cellular
exact sequence \ $0\arr K\arr G\arr M\arr 0$ and $\size{G}
=\k^{\aln}$. If $\k=\k^{\aln}$, then all members of the cellular
exact sequence have the same size $\k$.
\end{corollary}

We note that applications of the (Strong) Black Box in \cite{BD}
provide cellular covers of size greater than or equal to the
continuum. Using a classical idea due to A.L.S. Corner (see
\cite{C63}) we will be able to derive cellular covers of size
$\kappa < 2^{\aleph_0}$ (which complements the results in Buckner
and Dugas \cite{BD}). On top of this, thanks to the preliminary
work in \cite[p. 16, Theorem 1.1.20]{GT}, our construction is much
simpler.

\begin{theorem} \label{crucial-kernels}
Let $R$ be a commutative, torsion-free, reduced  ring with $1$, of
size $<\Cont$. Let $K$ be any torsion-free and reduced $R$-module
of rank $\k < \Cont$. Then there is a cotorsion-free $R$-module
$G$ of rank $3$ if $\k=1$, and of rank $3\k+1$ if $2\le \k<
\Cont$, with submodule $K$ such that $\Hom(G,K)=0$ and
$\Hom(G,G/K) =\pi R$ where $\pi: G\arr G/K \ (g\mapsto g+K) $ is
the canonical epimorphism. In particular,
$$
0\arr K\arr G\arr G/K\arr 0
$$ is a cellular exact sequence.
\end{theorem}


 \Pf We first note that the assumptions on $R$ and $K$ imply
that $R$ and $K$ are cotorsion-free, respectively by \cite[p. 19,
Corollary 1.1.25]{GT}. Also recall that any ordinal $\a$  is the
same as the set $\{\beta\mid \beta < \a\}$, in particular any
natural number $k$ is  $k =\{0,1,\dots, k-1\}$. Choose a free
$R$-submodule $E=\bigoplus_{\alpha < \k}Re_\a \subseteq K$ of
rank $\k$, such that $K/E$ is torsion, and let $F=\bigoplus_{\a <
\k} Rf_\a $ be a free $R$-module of the same rank. Note that $E$
also exists if $\rk(K)=|K|$. If $C= K\oplus F$, then we define the
$R$-module $G$ as a pure submodule of the completion $\Chat$. We
distinguish two cases: If $1 < \k$ is finite, then we let
$f'=f_0+...+f_{\k-1}$ and define
 \begin{eqnarray} \label{k1}
 G=\langle K, F, w_\a(w e_\a  + f_\a), w'f' \mid \a < \k \rangle_*\subseteq \Chat
 \end{eqnarray}
where $w$, $w'$ and $w_\a \in \widehat R$,  $(\a< \k)$, is a family of
algebraically independent elements over $C$. Its existence
follows from a theorem of G\"obel-May; see \cite[p. 16, Theorem
1.1.20]{GT}. If $\k=1$ we can omit the element $w'f'$ (as the
argument (\ref{k8}) below will not be needed).

In case $\k$ is infinite, we define
 \begin{eqnarray}\label{k2}
 G=\langle K, F, w_\a(w e_\a  + f_\a), w'_\a(f_0+f_\a) \mid \a<  \k \rangle_*\subseteq \Chat
 \end{eqnarray}
where $w$, $w'_\a$ and $w_\a$ $(\a<\k)$, is again a family of
algebraically independent elements over $C$; for its existence we
can apply the same result as above because $\k < \Cont$.

Clearly the rank of $G$ is $3$ for $\k=1$, $3\k +1$  if $1 < \k$ is
finite and $\k$  if $\k$ is infinite. We now must show that the exact
sequence $0\arr K\arr G\arr M\arr 0$ with $M=G/K$ satisfies the
conditions stated in the theorem.

Let $\pi: \Chat=\Khat \oplus \Fhat \arr \Fhat$ be the canonical
projection with kernel $\Khat$.

We consider the case when $\k\ne 1$ is finite. The
infinite case is similar and left to the reader. The case $\k=1$ is
trivial.

If $x\in G$, then by (\ref{k1}) there is $s\in \bS$ such that
 \begin{eqnarray}
 \label{k3} sx= k + f + \sum_{\a} r_\a w_\a (w e_\a + f_\a)   +r'w'f'
 \text{ for some } k\in K, f\in F, r', r, r_\a \in R. \end{eqnarray}

By continuity of $\pi$ it follows that $\pi (s x) = f + \sum_{\a}
r_\a w_\a f_\a   +r'w'f' $, thus
 \begin{eqnarray} \label{k4}
 \pi (G) \subseteq \langle F, w_\a f_\a  , w'f'\mid \a < \k  \rangle_*\subseteq \Fhat . \end{eqnarray}

Next we show that
 \begin{eqnarray} \label{k5} G\cap \Khat = K. \end{eqnarray}

It will suffice to verify $G\cap \Khat \subseteq K$. If $x\in
G\cap \Khat$, by (\ref{k1}) there is $s\in \bS$ such that $sx=k +
f + \sum_{\a} r_\a w_\a(w e_\a  + f_\a)   + r'w'f'= k' \in \Khat$.
Thus
$$
[k+ \sum_{\a} r_\a w_\a w e_\a - k' ]  +  [f + \sum_{\a} r_\a w_\a f_\a
  +r'w'f'] = 0.
$$
Since the sum $\widehat K \oplus \widehat F$ is direct we have
$k+\sum_{\a} r_\a w_\a w e_\a  - k'  =  f + \sum_{\a} r_\a w_\a f_\a
 + r'w'f'=0$. Algebraic independence on the second term
implies $f=0, r'=0$ and $r_\a=0$ for all $\a < \k $. From the
first term we get $k-k'=0$, so $sx= k'=k\in K$. Purity and
torsion-freeness of $K$ implies $x\in K$ as required.

\bigskip

We now show that
\begin{eqnarray} \label{k6} \Hom(G,K)=0. \end{eqnarray}

Let $\va: G\arr K$ be a  homomorphism. By continuity of $\va$ and
$w_\a (w e_\a  + f_\a) \in G$ we get $\va (w_\a (w e_\a  + f_\a ))
= w_\a(w \va (e_\a)   + \va (f_\a ))  \in K$, hence  $w_\a ( w \va
(e_\a)  + \va (f_\a))  = k'$ for some $k'\in K$. Again by
algebraic independence of $w_\a$ we get $\va (e_\a)=0$ and $\va
(f_\a)=0$ for all $\a < \k$. Thus $\va (E)=\va (F)=0$, so that
$\va$ induces a map $\va': G/(E\oplus F)\arr K$. The torsion part
of $G/(E\oplus F)$ is $(E\oplus F)_*/(E\oplus F)$. It must vanish
under the induced map $\va'$, because $K$ is torsion-free. Thus
$\va$ factors through $G/(E\oplus F)_*$ which is divisible while
the image is reduced, so all of $G$ is in the kernel and $\va=0$
as claimed in (\ref{k6}).

Finally we show
\begin{eqnarray} \label{k7} \Hom(G,G/K)= \pi R. \end{eqnarray}

By (\ref{k4}) and (\ref{k5})  it follows that $\pi (G)=\langle F,
w_\a f_\a  , w'f'\mid \a< \k \rangle_* = G/K$ canonically. Thus we
can view $\va: G\arr G/K$ as the map $$\va: G\arr \langle F, w_\a
f_\a, w'f'\mid \a < \k\rangle_* \subseteq \Chat,$$ and also
$G\subseteq \Chat$.

For any $\beta <\k $ we have $ s w_\beta(w \va (e_\beta)   + \va (f_\beta))
= s \va (w_\beta (w e_\beta  + f_\beta ))  = f + \sum_{\a}
r_\a w_\a f_\a   +r'w'f'$ for suitable coefficients. By algebraic
independence and torsion-freeness follows $f=0$, $r_\a=0$ $(\a \ne
\beta)$, $r'=0$ and $s w_\beta (w \va (e_\beta)   + \va (f_\beta) )  =
r_\beta w_\beta f_\beta   $. Thus, $s \va(w e_\beta)  = r_\beta f_\beta
 - s\va (f_\beta) $. Now $\va (f_\beta) \in \langle F, w_\a f_\a
, w'f'\mid \a \le \k \rangle_*$, and by algebraic independence of
$w$ also $\va (e_\beta) =0 $ and $\va (sf_\beta) =r_\beta f_\beta
$ for all $\beta < \k $. In particular $E\subseteq \Ker\va$, so
$\va\restr K$ induces $\va': K/E\to G/K$. However, $K/E$ is
torsion while $G/K$ is torsion-free, hence $\va'=0$.  This shows
that $K\subseteq \Ker \va$, and  $\va$ factors through $\varphi:
G/K\to G/K$. But we have seen before that $\va (f_\beta) =r_\beta
f_\beta $ (w.l.o.g. we may assume $s=1$ by purity and
torsion-freeness). If we apply $\va$ to $w'f'$ we will obtain
similarly $\va( f')=r'f'$ for some $r'\in R$ and derive
\begin{eqnarray} \label{k8}
r'f_0+...+r'f_{\kappa-1}=r'f'=\va (f') =\va (f_0)+...+\va
(f_{\kappa-1})=r_0f_0+...+r_{\kappa-1}f_{\kappa-1}
\end{eqnarray}
By linear independence follows $r_\beta =r'$ for all $\beta < \k $
and $\va = \pi r'  \in \pi R $ as desired. \Fp
\bigskip

We would like to remark that the rank of $G$ is chosen minimal
(as stated in the theorem).


\begin{example}
{\rm For $K=\Z$ this yields a cellular cover $ \Z \to G \to M$ with $G$
and $M$ of rank $3$ and $2$, respectively. Here is an explicit presentation of $G$:
Let $R=\Z$, $\bS=\{1,p, p^2, ...\}$ for a given prime $p$,
$\widehat R$ the ring of $p$-adic integers,
then $G=\langle e, f, w(w' e  + f) \mid \rangle_*$
where $w$ and $w'$ are two linear independent $p$-adic numbers, not integers.

This is the simplest example of a cellular cover of abelian
groups with kernel $\Z$. In
fact, it will be shown in \cite{RS} that any group of rank $1$ does
not admit cellular covers with free kernel.
}
\end{example}

On the other hand, the construction of Theorem~\ref{crucial-kernels}
can be easily modified so that the rank of $G$ becomes larger.
For this, one can take $F$ free of rank $\kappa'$,
with $\kappa<\kappa'<\Cont$, and
$$
G=\langle K, F, w_\beta(w e_\a  + f_\beta), w'f' \mid \a < \k \rangle_* \mid \a<  \k, \beta<\kappa' \rangle_*\
$$ where $\alpha= \psi(\beta)$ is any fixed surjective function
$\psi: \kappa'\to \kappa$ (cf. Theorem~\ref{crucial-quotients}).

We end this section with an interesting case of Theorem~\ref{crucial-quotients}
not obtained in Buckner and Dugas's paper \cite{BD09}.
\begin{corollary} If $K$ is a torsion-free, reduced $R$-module  of infinite size
$\k <\Cont$, then there is a cellular exact sequence $0\arr K\arr
G\arr M\arr 0$ with $\rk(G)=\rk(M) =\k$. \qed
\end{corollary}

\section{Two theorems about cokernels of cellular covers}\label{cokernl}

In this section we take the opposite point of view and want to
prescribe certain cotorsion-free modules $H$ such that $0 \arr
K\arr G\arr H\arr 0$ is a cellular exact sequence for suitable
$0\arr K\arr G$'s. Our method will work for particular rings $R$
and $\bS$-topologies. Recall that an $R$-module $M$ is called
{\it $\aleph_0$-free} if every finite rank submodule of $M$ is
contained in a pure free submodule of $M$. For a cardinal
$\kappa$ let $E=\bigoplus\limits_{\alpha < \kappa}R e_{\alpha}$
be a free module of rank $\kappa$. It will be clear from the
context to which cardinal $\kappa$ the module $E$ refers to.\\

As in Section $3$ we have a borderline $2^{\aleph_0}$ and must
distinguish between cokernels of size below or above and equal to
$2^{\aleph_0}$. Moreover, due to technical reasons in the
construction we will have to assume that $H$ is rigid and that $H$
has no non-trivial homomorphism into any $\aleph_0$-free
$R$-module (in the Black Box construction). By Proposition
\ref{endo} this is a reasonable restriction on $H$.

\subsection{Cokernels of size below the continuum}

The main result of this section reads as follows
\begin{theorem} \label{crucial-quotients} Let $\kappa' < \Cont$ be a cardinal, $R$ a commutative, torsion-free ring with $1\ne 0$, of size $< 2^{\aleph_0}$. Moreover, let $H$ be a cotorsion-free $R$-module which is $\kappa'$-generated such that $\End(H)=R$. Then for any cardinal $\k$ with $\k'\le \k <\Cont$ there is a
cotorsion-free $R$-module $G$ with the following properties.
\begin{enumerate}
\item $G$ is $\kappa$-generated and there is $K\subseteq G$ with $G/K=H$.
\item If $\va\in \End(G)$, then there is a unique $r\in R$ such that $(\va-r\cdot \id_G)(K) =0$,
so there is an induced homomorphism  $\va_r: H \arr G \ ((g + K) \mapsto (\va-r)g)$.
\item If $\Hom(H,G)=0$, then $\End(G) =R$.
\item $\Hom(G,H) =\pi R$ where $\pi: G\arr H \ (g\mapsto g+K) $ is the
canonical epimorphism.
\end{enumerate}
In particular, if \ $\bS$ is chosen such that $\Hom(H,G)=0$, then
 $0 \arr K\arr G\arr H\arr 0$ is a cellular exact sequence.
\end{theorem}

Note that for $R$-modules of size $<2^{\aleph_0}$
cotorsion-freeness is equivalent to say that the $R$-module is
torsion-free and reduced; see \cite{GT}.

 \Pf We enumerate a generating system $\{h_\beta \mid \beta < \kappa'\}$ of $H$ and choose a surjection
$\psi: \kappa\to \kappa'$. If $C= E \oplus H$, then we define the
desired module $G$ as a pure submodule of the completion $\Chat$, as follows:
 \begin{eqnarray}
 \label{q1}
 G=\langle E, w_\a e_\a,
 w'_\a(e_1 + e_\a), (\tilde{w}_\a e_\a  + h_\beta) \mid
 \alpha < \kappa, \psi(\alpha)=\beta \rangle_*\subseteq \Chat,
 \end{eqnarray}
where  $w_\alpha$, $w'_\alpha$, $\tilde{w}_\alpha \in \hat{R}$
$(\alpha < \kappa)$ is a family of algebraically independent
elements over $C$. Using $|C| < 2^{\aleph_0}$ its existence
follows again from a theorem of G\"obel-May; see \cite[p. 16,
Theorem 1.1.20]{GT}.

Clearly $G$  is $\kappa$-generated as required. We also consider
the projection induced by the decomposition $\Chat = \Ehat \oplus
\Hhat$, which is
$$\pi: \Chat \arr \Hhat \text{ with kernel } \Ehat.$$

We will not  distinguish between $\pi$ and its natural
restrictions to submodules of $\Chat$. If $x\in G$, then there is
$s\in \bS$ such that
\begin{eqnarray}
\label{q2} sx= f + \sum_{\alpha < \kappa} r_\alpha w_\alpha e_\alpha
+ \sum_{\alpha < \kappa} r'_\alpha w'_\alpha (e_1 + e_\alpha)  +
\sum_{\alpha < \kappa,\, \psi(\alpha)=\beta}\tilde{r}_\alpha (\tilde{w}_\alpha e_\alpha
 + h_\beta)
\end{eqnarray}
for some  $f \in E, r_\alpha, r'_\alpha, \tilde{r}_\alpha \in R$.

By definition and continuity of the map $\pi$ it now follows that
$$\pi (sx)= \sum_{\alpha < \kappa,\, \psi(\alpha)=\beta} \tilde{r}_\alpha h_\beta,$$ thus
 \begin{eqnarray} \label{q3} \pi (G)= H  \end{eqnarray}
since $H$ is pure in its completion $\widehat{H}$ and we let
$K:=\Ehat \cap G$. Obviously $K=\Ker(\pi \restr  G)$ and
thus
 \[ E \subseteq K = \{ x \in G \mid \sum_{\alpha < \kappa, \, \psi(\alpha)=\beta}
\tilde{r}_\alpha h_\beta
=0 \mbox{ in equation  (\ref{q2})} \}. \]
This implies that $K/E$ is divisible and the quotient $(G/E)/(K/E) \cong G/K \cong H$ is reduced. There is a decomposition into a divisible summand and a reduced part $H$:
\begin{eqnarray} \label{q4} G/E = (K/E) \oplus H \text{ with } K/E \text{ the maximal divisible summand.} \end{eqnarray}
 Next we establish (ii) and study $\varphi \in \End(G)$. If $\delta < \kappa$,
$\gamma=\psi(\delta)$, then there is $s=s_\delta \in S$ such that

$$ s\varphi(w_\delta e_\delta )=s w_\delta \varphi(e_\delta)= f + \sum_{\alpha < \kappa} r_{\alpha} w_{\alpha} e_{\alpha}   +
\sum_{\alpha < \kappa} r'_{\alpha} w'_{\alpha} (e_1 + e_{\alpha}) +
\sum_{\alpha < \kappa,\,\psi(\alpha)=\beta}\tilde{r}_{\alpha}(\tilde{w}_{\alpha} e_{\alpha} +
h_\beta)$$ for some $f \in E$ and $r_{\alpha},
r'_{\alpha}, \tilde{r}_{\alpha} \in R$.

Again, the sum $\Ehat \oplus
\Hhat$ is direct and $w_{\alpha},w'_{\alpha}, \tilde{w}_{\alpha}$
$(\alpha < \kappa)$ are algebraically independent elements over $C$.
Therefore, equating coefficients we obtain
$s\varphi(e_{\delta})=r_{\delta} e_{\delta}$ and $\varphi$ acts on
$e_{\delta}$ as multiplication by $r_{\delta}$.
A similar argument shows that $\varphi$ also acts on $(e_1 +
e_{\delta})$ as multiplication by some $r'_{\delta}\in R$, for all $\delta < \kappa$. Therefore,
\[ r'_\delta (e_1 + e_\delta)=\varphi(e_1 + e_\delta) =
\varphi(e_1) + \varphi(e_\delta) = r_1 e_1 + r_\delta e_\delta. \]
And comparing components we get $r_\delta=r$ for all $\delta <
\kappa$ (which does not depend on $f$). It follows that $\varphi
\restr E = r\cdot \id$. Using that $G$ is reduced and (\ref{q4}) we
see  that $\varphi - r \cdot \id$ induces a unique homomorphism
$\va_r: H=G/K \rightarrow G$, which shows (ii).

By the assumption of (iii) we note that the induced map $\va_r$ from (ii)
vanishes, thus $\End(G)=R$ in this case and also $\Hom(G,K)=0$.\\

For (iv) we consider any  $\psi \in \Hom(G,H)$. Clearly
\[ \psi(w_{\alpha} e_{\alpha} ) =
w_{\alpha} \psi(e_{\alpha})\in H \cap w_{\alpha} H  \] which is zero
by the algebraically independent element $w_\a$. Hence $\psi$ induces
a homomorphism  $\tilde{\psi}: G/E \rightarrow H$. Since $H$ is
reduced it follows from (\ref{q4}) that $\tilde{\psi}(K/E)=0$ and we
can write $\psi =\pi r$ for some $r \in R$ and $\psi \in \pi R $. The
reverse inclusion for (iv) is trivial. Finally note that $0\arr K\arr
G \arr H\arr 0$ is a cellular exact
sequence by the properties of the maps just studied. \qed\\

Note that there are many examples of rings $R$ and
$\bS$-topologies such that the constructed module $G$ in the above
theorem satisfies $\Hom(H,G)=0$ for given $H$. For instance, if
$H$ is a rank one group that is not a ring, then one could choose
$R=\Z$ and the $\Z$-adic topology on $R$. Consequently, $G$ will
be $\Z$-homogeneous and hence $\Hom(H,G)=0$. This shows
Corollary~5.5 from \cite{FG}.

\subsection{Cokernels of size greater than or equal to the
continuum} We now consider cellular covers of cotorsion-free modules of size $\geq
2^{\aleph_0}$ and utilize the Black Box from \cite{CG} to extend
Theorem \ref{crucial-quotients} to larger cardinals.

Let $\lambda > |R|$ be an infinite cardinal such that $\lambda =
\lambda^{\aleph_0}$. If $\mu$ is any infinite cardinal (and
$\mu^{\aln} >\size{R}$), then $\mu^{\aleph_0}$ is a candidate for
$\lambda$. The cardinal condition ensures that the set of all
countable subsets of $\lambda$ has size $\lambda$ as well. This will
be used to get that the completion of our canonical free base module
$B$ of size $\lambda$ has size $\lambda$ as well. The heart of the
Black Box construction is to build the desired $R$-module on a tree
$T = {}^{\omega > } \lambda $ as its underlying set of `supports'
using the additional geometric structure.\\

Therefore let $T$ be the set of all finite sequences
 $\tau = {\lambda_0}^\wedge \cdots {}^\wedge \lambda_{n-1}$ in
$\lambda$, hence
$$T = \{ \tau : n \arr \lambda : n \in \omega \}. $$
Recall that $\tau$ above is a finite branch of length $n$,
thus $\dom(\tau)= n=\{0,\dots, n-1\}$. Similarly we define
the set $\Br(T)$ of all infinite branches of length $\omega$, which
is
$$\Br(T) = {}^\omega \lambda = \{ v: \omega \arr \lambda \}.$$
This set has cardinality $\lambda$ by assumption on $\lambda$.
Finite and infinite branches $v$ have a canonical {\it support }
which is the subset
$$[v] = \{v\restr m \in T: m  < \mbox{length of $v$}\}$$
of $T$. Note that $[v],[w]$ are almost disjoint, that is $[v]\cap [w]$ is finite,
if and only if $v,w$ are distinct branches.
Trees also have a natural ordering by
extensions as follows: For any $\tau, \nu\in T$
$$ \tau < \nu \iff \tau \subseteq \nu \iff
[\tau] \subseteq [\nu] \mbox{ and } \nu \restr \dom \tau = \tau.
 $$
The {\it norm} of a branch $\tau\in Br(T)$ is defined as
$$||\tau||= \sup (\Im(\tau))\in \lambda.$$

We transport these supports and norms to an $R$-module, taking
$$E = \bigoplus\limits_{\tau \in T} R \tau$$
to be the free $R$-module generated by $T\subset E$ (where $\tau
\in T$ is identified with $1 \tau  \in E$). Any element $g \in
\hE$ can be expressed as a countable sum $g = \sum_{n \in \omega}
g_n \tau_n$ for some $g_n \in \widehat {R}$,
such that for all $m\in \omega$, $g_n \in
q_m\widehat {R}$ for almost all $n \in \omega$. We denote by $[g]
= \{\tau_n : g_n \neq 0 , n \in \omega \}$ the {\it support } of
$g$. If $v$ is an infinite branch, then we also denote by $v =
\sum_{n \in \omega } q_n (v\restr n) \in \hE$ and call this element
a {\it branch-element}, which obviously has the same support as the
branch $v$, namely $[v]$. These branch-elements will be useful tools
to recognize elements of the module $G$ under construction.\\

Let $H$ be a cotorsion-free $R$-module of size less than $\lambda$
and put $B:=E \oplus H$. As before let $\pi: \hB \rightarrow \hH$
be the canonical projection onto $\hH$. As in the previous section
our goal is to construct $G \subset_* \hB$ such that $\pi(G)=H$,
$\Hom(G,H)= \pi R$ and $\Hom(G,K)=0$, where $K=\hat{E} \cap
G=\Ker(\pi\restr G)$. Note that in this case $0 \rightarrow K
\rightarrow G \rightarrow H \rightarrow 0$ is a cellular cover
over $H$. In order to ensure these properties we need to satisfy
two requirements during the construction: $G \cap \hH=0$ and
$\End(G)=R$. Thus we will have
\begin{eqnarray}\label{sandw} E \subseteq_*
G \subseteq_* \hB
\end{eqnarray} but $H$ must not be contained in $G$. This
forces us to adjust Shelah's Black Box to (\ref{sandw}). Thus, if
$g+h \in \hB=\hE \oplus \hH$, we let $[g+h]:=[g]$ be the support
of $g+h$ just looking at the $\hE$-component of elements.

The notions of support and norm extend naturally to subsets $X$ of
$\hB$ as follows: $$[X] = \bigcup \{ [x] : x \in X \}$$ and
$$||X|| = \sup \{ ||\tau|| : \tau\in [X]\}.$$
Note that $||X||$ is an ordinal which is strictly less than
$\lambda$ if $[X]$ is countable, because the cofinality
$\cf(\lambda )$ of $\lambda$ is greater than $\aleph_0$.

The classical Black Box also needs the notion of traps which are
partial approximations to endomorphisms $\hB \to \hB$. In our
case, we will approximate homomorphisms of the form $\varphi : \hE
\oplus H \arr \hB$.

As before let ${}^{\omega > }\omega$ be the countable tree of
finite sequences in $\omega $ and let $f: {}^{\omega > }\omega
\arr T$ be a tree embedding. Let $\varphi: \dom(\varphi)
\rightarrow \hB$ denote a partial homomorphism with countable
domain $\dom(\varphi) \subseteq_* \hB$ a pure submodule of $\hE
\oplus H$ and suppose that $[\dom(\varphi) ] $ is a countable
subtree of $T$. We will
require

${} \hfill \Im f \subseteq [\dom(\varphi)] \subseteq T   \hfill
(f,\varphi)$

\noindent and call $(f,\varphi)$ a {\it trap}.

By our assumptions on the cardinal $\lambda$ it is clear that the
number of traps is $\lambda$, hence the following theorem is an easy
modification of Shelah's classical Black Box, see the
appendix of \cite{CG}. Recall that for an ordinal $\rho$ we let
$\rho^o=\{ \delta < \rho : \cf(\delta)=\omega \}$.

\begin{theorem}[Shelah's Black Box]
\label{BB} Let $\lambda =  \lambda^{\aleph_0} \geq |R|$ be an
infinite cardinal
 and $T = {}^{\omega > } \lambda $ be a tree which is the basis of a free
$R$-module $B = \bigoplus\limits_{\tau \in T} R \tau$. Moreover,
let $H$ be a cotorsion-free $R$-module of size less than $\lambda$
and let $\rho=\cf(\lambda)$. For any choice of disjoint stationary
subsets $S_1, S_2 \subseteq \rho$ there exists an ordinal
$\lambda^*$ of cardinality $\lambda$ and a list of traps
$$ (f_\alpha, \varphi_\alpha), \ \ \alpha \in \lambda^*$$
with the following properties:
\begin{itemize}
\item[(a)]  $||\Dom(\varphi_\alpha)|| \in \rho^o$ is a limit ordinal
with
$||v|| = ||\Dom(\varphi_\alpha)||$ for all $v \in \Br(\Im f_\alpha)$.
\item[(b)] If $\beta < \alpha \in \lambda^*$ then
$||\Dom(\varphi_\beta)|| \leq ||\Dom(\varphi_\alpha)||$ and
$\Br(\Im f_\beta)\cap \Br(\Im f_\alpha) =
\emptyset$.
\item[(c)]  If $\beta + 2^{\aleph_0} \leq \alpha $, then
$\Br(\Dom(\varphi_\beta)) \cap \Br(\Im f_\alpha) = \emptyset$.
\item[(d)] (PREDICTION) If $X$ is a countable subset of $\hE
\oplus H$ and $\varphi : \hE \oplus H \arr \hB$ is a homomorphism,
then there exist ordinals $\alpha_1, \alpha_2 \in \lambda^*$ such
that $X \subseteq \Dom(\varphi_{\alpha_i})$ and $\varphi \restr
\Dom( \varphi_{\alpha_i}) = \varphi_{\alpha_i}$ and $||
\Dom(\varphi_{\alpha_i})|| \in S_i$ for $i=1,2$.
\end{itemize}
\end{theorem}

In the classical Black Box it is used that mappings from a free
$R$-module $B$ to $B$ have unique extensions to mappings from
$\hB$ to $\hB$. We still want to argue with unique extensions of
mappings which explains the following observation.

\begin{lemma}
\label{extension} Let $\varphi: \dom(\varphi) \rightarrow \hE
\oplus H$ be such that $E' \subseteq_* \dom(\varphi) \subseteq_*
\hE \oplus H'$ and $\pi(\dom(\varphi))=H'$ for some submodule $H'
\subseteq_* H$ and direct summand $E' \subseteq E$. Then $\varphi$
has a unique extension $\widehat{\varphi} : \hE' \oplus H'
\rightarrow \hE \oplus \hH=\hB$.
\end{lemma}

\proof Since the completion commutes with finite direct sums we may
assume without loss of generality that $E'=E$. Moreover, since
$\dom(\varphi)$ is pure in $\hE \oplus H'$ it follows that $E
\subseteq_* \dom(\varphi) \cap \hE$ is pure in $\hE$. Hence there is
a unique map $\hv_{E} : \hE \rightarrow \hE \oplus \hH=\hB$ such that
$\hv_{E} \restr {\dom(\varphi) \cap \hE} = \varphi \restr
{\dom(\varphi) \cap \hE}$. Let $h \in H'$. Then $\hb + h \in
\dom(\varphi)$ for some $\hb \in \hE$ since $\pi(\dom(\varphi))=H'$.
Define
\[ \hv_{H'} : H' \rightarrow \hE \oplus \hH=\hB \textrm{ via
} h \mapsto \varphi(\hb + h) - \hv_{E}(\hb).
\]
We claim that $\hv_{H'}$ is a well-defined homomorphism. Assume
first that there are $\hb_1, \hb_2 \in \hB$ such that $\hb_1+h$
and $\hb_2+h$ are in $\dom(\varphi)$. Then
\[ \varphi(\hb_1 +h)-\hv_E(\hb_1) -
(\varphi(\hb_2+h)-\hv_E(\hb_2)) = \]
\[ =\varphi(\hb_1+h-(\hb_2+h)) - \hv_E(\hb_1-\hb_2) =
\varphi(\hb_1-\hb_2) - \hv_E(\hb_1-\hb_2) = 0 \] since
$\hb_1-\hb_2 \in \dom(\varphi) \cap \hE$. Hence $\hv_{H'}$ is well-defined.\\
Now assume that $h_1,h_2 \in H'$ and $\hb_1,\hb_2,\hb_3 \in \hB$ are
such that $\hb_1+h_1, \hb_2 + h_2, \hb_3 + h_1+h_2 \in
\dom(\varphi)$. It follows that
\[ \hv_{H'}(h_1+h_2) - \hv_{H'}(h_1) - \hv_{H'}(h_2) = \]
\[= \varphi(\hb_3-\hb_2-\hb_1) - \hv_E(\hb_3-\hb_2-\hb_1) =0 \]
since again $\hb_3-\hb_2-\hb_1 \in \dom(\varphi) \cap \hE$. Hence
$\hv_{H'}$ is a homomorphism.\\ Define $\hv : \hE \oplus H'
\rightarrow \hE \oplus \hH$ by $\hv(\hb + h):=\hv_E(\hb) +
\hv_{H'}(h)$. Then
clearly $\hv$ extends $\varphi$.\\
Finally, assume that $\hp$ extends $\varphi$ too. Therefore
$\hv\restr E=\hp\restr E=\varphi \restr E$, and
hence uniqueness of $\hv_E$ implies that $\hv_E = \hp
\restr {\hE} = \hv \restr {\hE}$. If $h \in H'$, then
$\hb+h \in \dom(\varphi)$ for some $\hb \in \hE$ and hence
\[ \hv(\hb +h)=\hv_E(\hb) + \hv_{H'}(h)=\varphi(\hb + h)\] and
similarly, \[ \hp(\hb+h)=\hv_E(\hb)+\hp(h)=\varphi(\hb+h).\]
Therefore $\hp(h)=\hv_{H'}(h)$ and so $\hv$ and $\hp$ coincide on
$\hE$ and on $H'$ and are thus identical. Uniqueness of $\hv$
is shown.\qed\\

We now want to apply the Black Box to show the following

\begin{theorem}\label{main2}  Let $R$ be a commutative, torsion-free ring with $1\ne 0$ and $\k$ a cardinal with $\kappa^{\aleph_0} > |R|$. Moreover, let $H$ be cotorsion-free $R$-module such that $\End(H)=R$ and $|H| \leq \kappa$. Then there is a
cotorsion-free $R$-module $G$ of size $\size{G} = \k^{\aln}$ with the following properties.
\begin{enumerate}
\item There is a submodule $K\subseteq G$ with $G/K=H$. \item If
$\va\in \End(G)$, then there is a unique $r\in R$ such that
$(\va-r\cdot \id_G)(K) =0$, so there is an induced homomorphism
$\va_r: H \arr G \ ((g + K) \mapsto (\va-r)g)$. \item If
$\Hom(H,M)=0$ for every $\aleph_0$-free module $M$, then
$\End(G)=R$, and $\Hom(G,K)=0$. \item $\Hom(G,H) =\pi R$ where
$\pi: G\arr H \ (g\mapsto g+K) $ is the canonical epimorphism.
\end{enumerate}
In particular, if $Hom(H,M)=0$ for all $\aleph_0$-free modules $M$ then
$0 \arr K\arr G\arr H\arr 0$ is
a cellular exact sequence.
\end{theorem}

\proof

Let $\lambda:=\kappa^{\aleph_0}$ and $E$, $T$ and $B$ be as above.
Then $\lambda^{\aleph_0}=\lambda$ is as required for the Black
Box. The module will be the union $G =\bigcup_{\alpha \in \lambda}
G_\alpha $ of an ascending, continuous chain of cotorsion-free,
$\aleph_0$-free modules  $G_\alpha$ $$E \subseteq_* G_\alpha \subseteq_* \hB
\mbox{ and } \pi(G_{\alpha}) \subseteq H.$$

Recall that $\pi: \hB \rightarrow \hH$ is the canonical projection.
Let $S_1$ and $S_2$ be two disjoint stationary subsets of $\rho^o$
where $\rho=\cf(\lambda)$ and such that the Black Box \ref{BB} holds
for $S_1,S_2$. We begin with $G_0 \subseteq_* \hE \oplus H$ such that
$G_0$ is free with $\pi(G_0)=H$. It is easy to construct $G_0$ by
adding elements of the form $v_h+h$ to $E$ where $h \in H$ and the
$v_h$'s form a set of almost disjoint infinite branches (see also
below for more details). By continuity we only need to consider the
inductive step at any $\alpha \in \lambda^*$. We want to find
$g_\alpha \in \hE \oplus H \subseteq \hB$ such that
$$G_{\alpha + 1} = \langle G_\alpha, g_\alpha \rangle_*$$
and the new element $g_\alpha$ must fulfill several tasks.

First we require that the new element $g_\a$ is a 'branch-like'
element of $\hB$: Recall that any branch $v \in \Br(\Im f_\alpha)$ has norm
$||v||=||\Dom(\varphi_\alpha)||$, by (a) of Black Box \ref{BB}, and
gives rise to a non-constant, branch-element, which we also denoted by $v$. Any
sum $b + v$ with
$$b \in \widehat{\Dom(\varphi_\alpha)} \mbox{ and } ||b|| < ||v||$$
is called a {\it branch-like element}. The point is that a
branch-like element has a support which at the top looks like a
branch from $\Br(\Im f_\alpha)$. The choice of generators for $G$
helps to describe the action of homomorphisms on $G$.  We are
ready for two preliminary Step-Lemmas. The first one will be
used to ``kill'' non-inner endomorphisms,
and the second to ``kill'' homomorphisms $\varphi: G\to H$
with no $\widetilde \varphi:H\to H$ such that $\varphi=\widetilde\varphi \pi$.

\begin{lemma}\label{killing-one} Let $G$ be the module constructed so far with
\begin{enumerate}
\item $E \subseteq_* G \subseteq_* \hB$
\item $\pi(G):=H' \subseteq H$.
\end{enumerate}
Let $\varphi \in \End(G)$ such that $\varphi \restriction E \not
\in R$. For any $0 \not= h \in H'$ there is an $x' \in \hE$ such
that $x=x' + h$ satisfies
$$ \varphi (x) \not\in \langle G, x \rangle_*.$$ Moreover, the module $G'=\left< G, x
\right>_*$ is cotorsion-free, $\aleph_0$-free, and $\pi(G')
\subseteq H$.
\end{lemma}

\proof First note that $\varphi$ has a unique extension (again
denoted by $\varphi$) to $\varphi: \hE \oplus H' \rightarrow \hB$
by Lemma \ref{extension} which is again not in $R$. Thus $\varphi
(x)$ is a well-defined element in $\hB$ for any choice of $x \in
\hE \oplus H'$. First we use the assumption on $\varphi$ to show
that there exists a countable subset
\begin{eqnarray} \label{D1}
 C \subseteq T
\mbox{ such that } D = \langle C \rangle \mbox{ satisfies }
(\varphi -r)\widehat{D} \not\subseteq G \mbox{ for all } r \in R.
\end{eqnarray}

Choose a `constant branch' $v: \omega \arr \{\eta \} $ at some
$\eta < \lambda$ and let $C'$ be a countable subset of $T$ such
that the branch-element $v$ belongs to $\widehat{D'}$ where $D' =
\langle C' \rangle$. If (\ref{D1}) fails for $D = D'$, there is $r
\in R$ with $(\varphi - r)\widehat{D'} \subseteq G$. This implies
that $(\varphi-r)\widehat{D'}=0$. Indeed, suppose
that for some $z\in \widehat D'$ we have $(\varphi -r) z \neq 0$, then, since  $\widehat R
\widehat D' \subseteq \widehat D'$, we would have a non-trivial homomorphism $\widehat R
\to G$ given by $\hat r \mapsto (\varphi -r) (\hat r z)$ which is impossible because
$G$ is cotorsion-free. Hence $(\varphi - r) \widehat D'=0$ as desired.

But
$\varphi\restriction E \not\in R$, then there is $x \in E$ such that
$(\varphi - r)x \not= 0$. Enlarge $C'$ to $C$ such that $x \in D =
\langle C \rangle$.
As before, using cotorsion-freeness of $G$ we have $(\varphi - r') \widehat D=0$.
Hence $(\varphi - r)v=0$ and $(\varphi - r')v=0$, which
implies $(r-r)v=0$, and thus $r=r'$ since $\widehat B$
is torsion-free. In particular $(\varphi -r)x=0$,
which is a contradiction.

Here is an alternative support argument to prove (\ref{D1}).
We include this at this point as similar arguments will be needed several times later on.
If (\ref{D1}) fails for $D$, then $(\varphi -
r')\widehat{D} \subseteq G$ for some $r' \in R$, hence $v(r-r')
\in G$. By construction there is $q_n$ and elements
$g_{\alpha} (\alpha \in I)$ for some finite index set $I$ such
that
\begin{eqnarray} \label{E0}
q_nv(r-r')=\sum\limits_{\alpha \in I}r_{\alpha}g_{\alpha} + e
\end{eqnarray} for some $r_{\alpha} \in R$ and $e \in E$. Now recall that by
construction all $g_{\alpha}$ are branch-like elements coming from
branches $v_{\alpha}$ that are not constant. Hence for every
$\alpha \in I$ there is some $m_{\alpha}$ such that $v
\restriction m_{\alpha} \not= v_{\alpha}\restriction m_{\alpha}$.
Let $m$ be the maximum of all these $m_{\alpha}$. Restricting the
equation (\ref{E0}) above to $v \restriction m$ it then follows
that $q_n(r-r') v \restriction m=0$ and hence $r=r'$. We derive
the contradiction $(\varphi - r)x = 0$ and (\ref{D1}) follows.

We may assume that
\begin{eqnarray} \label{D2}
 D \mbox{ in (\ref{D1}) also satisfies }
(q_n\varphi - r)\widehat{D} \not\subseteq G \mbox{ for all } n >
0, \ r \in R \setminus q_n R
\end{eqnarray}

Suppose (\ref{D2}) fails for $\psi = q_n\varphi - r$. Now we
choose elements $\sigma_m \in T$ of length $m$ which constitute an
`anti-branch', that is, two $\sigma_m$'s are incomparable in $T$.
Moreover we require
$$\sup_{m \leq k} || \psi(\sigma_m)|| < ||\sigma_{k+1}|| \mbox{ and let }
t=t_n := \sum_{m \in \omega} q_{m}q_n^m\sigma_m.$$ Then, for every
non-zero $k < \omega$, we have
$$\psi (t) \equiv \sum_{m \leq
k}q_mq_n^m \psi (\sigma_m) \mbox{ mod } q_{k}q_n^{k+1}\hB,$$ hence
$$\psi (t) \restr \sigma_k \equiv - q_{k}q_n^kr \mbox{ mod }
q_{k}q_n^{k+1} R$$ which cannot be $0$ because $r \not\in q_n R$.
Hence $||\psi (t)|| = \sup ||\sigma_k||$ and the element $\psi
(t)$ cannot be in $G$ by a support argument. Now we enlarge $D$ such that all
the $\sigma_m$'s belong to $D$ and this failure for $(r,q_n)$ is
impossible for the enlarged $D$. Similarly we deal with the other
(at most countably many) potential failures of (\ref{D2}) and
correct $D$, hence (\ref{D2}) holds.

We are now ready to find the desired element $x \in \hE \oplus
H'$. Let $h \in H'$. First we choose a new constant branch $w$
with $||D||, ||\varphi (D)|| < ||w||$. If the branch-like element
$x = w + h$ satisfies the lemma, then the proof is finished.
Otherwise $\varphi (x) \in \langle G, x \rangle_*$ and there are
$n \in \omega, r \in R$ such that
$$ q_n  \varphi (w+h)  - r (w+h)\in G.$$
If $n=0$ we can apply (\ref{D1}) directly, but if $n > 0$, we may
assume that  $r \in R \setminus q_n R$ and (\ref{D2}) applies as
well. (Note that $G$ is pure in $\hB$ and the $\bS$-topology is
Hausdorff.) There is a $d \in \widehat{D}$ such that $(q_n\varphi
- r)d \not\in G$. Now it is easy to check
 by support arguments that $x = w + (d+h)$ meets the requirements of the
lemma. \qed

\begin{lemma}
\label{killing-two} Let $G$ be the module constructed so far with
\begin{enumerate}
\item $E \subseteq_* G \subseteq_* \hB$
\item $\pi(G):=H' \subseteq H$.
\end{enumerate}
Let $\varphi \in \Hom(G,H)$ such that $\varphi \restriction E
\not= 0$. For any $0 \not= h \in H'$ there is an $x' \in \hE$ such
that $x=x' + h$ satisfies
$$ \varphi (x) \not\in H.$$ Moreover, the module $G'=\left< G, x
\right>_*$ is cotorsion-free, $\aleph_0$-free and $\pi(G')
\subseteq H$.
\end{lemma}

\proof As in Lemma \ref{killing-one} $\varphi$ has a unique
extension (again denoted by $\varphi$) to $\varphi: \hE \oplus H'
\rightarrow \hB$. Since $H$ is cotorsion-free the following is
obvious: There exists a countable subset
\begin{eqnarray} \label{D3}
 C \subseteq T
\mbox{ such that } D = \langle C \rangle \mbox{ satisfies }
\varphi(\widehat{D}) \not\subseteq H.
\end{eqnarray}

We find the desired element $x \in \hE \oplus H'$. Let $h \in H'$.
First we choose a new constant branch $w$ with $||D||, ||\varphi
(D)|| < ||w||$. If the branch-like element $x = w + h$ satisfies
the lemma, then the proof is finished. Otherwise $\varphi (x) \in
H$. But $\varphi \restriction E \not= 0$ and $H$ is
cotorsion-free, hence there is $e \in E$ such that $\varphi
(\widehat{R} e) \not\subseteq H$. Choose $\gamma \in \widehat{R}$
such that $\varphi (\gamma e) \not\in H$. Now it is obvious that
$x = w + (\gamma e+h)$ meets the requirements of the
lemma. \qed\\

We now  continue the construction of $G$ and describe the two tasks
depending on the traps $(\varphi_\alpha, f_\alpha)$ of the Black Box
in order to control endomorphisms of $G$ and homomorphisms $G\to H$:\\

(I) \ First we consider the following `bad case' for $\alpha$:\\
{\it If there is an $x \in \widehat{\dom(\varphi_\alpha)}$ such
that $||x|| < ||\dom(\varphi_\alpha)|| \in S_1 \cup S_2$ and
either
$$\varphi_\alpha (x) \not\in \langle G_\alpha , x \rangle_*
\mbox{\rm \quad ( if $||\dom(\varphi_\alpha)|| \in S_1$) \quad
or}$$
$$ \varphi_\alpha (x) \not\in H \mbox{ \rm (if $||\dom(\varphi_\alpha)|| \in S_2$)}$$
then choose a branch $v \in \Br(\Im f_\alpha)$ and put $g_\alpha =
v$ or $g_\alpha = x + v$; the choice of $g_\alpha$ depends on the
requirement that either}
$$  \varphi_\alpha (g_\alpha) \not\in \langle G_\alpha , g_\alpha \rangle_*
= G_{\alpha +1} \mbox{ ( \rm if
$||\dom(\varphi_\alpha)|| \in S_1$) \quad or }$$ $$ \varphi_\alpha
(g_\alpha) \not\in H \mbox{ \rm\quad (if $||\dom(\varphi_\alpha)||
\in S_2$)}.$$ {\it If $\alpha$ is not bad, then choose any
branch-like
$g_\alpha$ taking care of (II).}\\

(II) \  {\it If $\beta < \alpha$ was bad before and $
\varphi_\beta (g_\beta)  \not\in G_{\beta + 1}$ then we still want
$ \varphi_\beta
(g_\beta) \not\in G_{\alpha + 1}$.}\\

We first show

\begin{lemma}
There is a choice of $g_{\alpha}'s$ such that the two tasks (I)
and (II) are satisfied. \end{lemma}

\proof The work on condition (I) is put into Lemma \ref{killing-one}
and Lemma \ref{killing-two}, hence (II) must be verified. We apply a
short argument from \cite[p. 457]{CG} and restrict ourselves to the
case of endomorphisms of $G$ (i.e. $||\dom(\varphi_\alpha)|| \in S_1)$. The
case of homomorphisms into $H$ (if $||\dom(\varphi_\alpha)|| \in
S_2$) is similar but easier and therefore left to the reader. When
defining $g_\alpha =g_{\alpha,v} = x + v$ we have a free choice of
branches $v \in \Br(\Im f_\alpha)$ which we now use. If (II) is
violated for some $\beta = \beta_v$, then
\begin{eqnarray} \label{dist}
\varphi_\beta (g_\beta)  \in \langle G_\alpha, g_{\alpha,v}
\rangle_*, \mbox{ hence } \varphi_{\beta_v} (q_{k_v} g_{\beta_v})
- g_{\alpha,v} a_v \in G_\alpha.
\end{eqnarray}

A support argument and (c) from the Black Box show that
$$\beta_v < \alpha < \beta_v + 2^{\aleph_0}$$
and if $\beta_0$ is the least ordinal satisfying this inequality,
then, for all $v\in Br(\Im f_\alpha)$,
$$\beta_0 < \beta_v < \beta_0 + 2^{\aleph_0}.$$
By cardinalities, there are two distinct branches $v,w \in \Br(\Im f_\alpha)$ such
that $\beta_v = \beta_w $. Suppose $k_v \geq k_w$. Subtracting the
according expressions (\ref{dist}) we get
$$ r_v g_{\alpha,v} - \frac{q_{k_v}}{q_{k_w}} g_{\alpha,w} r_w \in G_\alpha$$
and by an easy support argument it can be seen that this is only
possible for $v = w$, a contradiction. Hence (II) can be arranged
for $g_\alpha = x + v$ and some $v \in \Br (\Im f_\alpha)$. \qed\\

We claim that the two tasks suffice to show the statement
of Theorem \ref{main2} and verify the conditions mentioned there:\\
First note that $K =\Ker(\pi)$, and $G/K$ is reduced. Consider any
$\varphi \in \End G$ and assume that $\varphi \restriction E
\not\in R$.

By the Lemma \ref{killing-one} there is an element $x \in \hE
\oplus H$ such that $\varphi (x) \not\in G$ and by the Black Box
we can find an $\alpha \in \lambda^*$ such that $\varphi$ extends
$\varphi_\alpha$, $x \in \widehat{\dom \varphi_\alpha}$ and $||x||
< ||\dom \varphi_\alpha||$. Hence $\alpha$ is a bad case and
$g_\alpha$ in the construction must satisfy (I) and (II). It
follows by task (I) that $ \varphi_\alpha (g_\alpha) = \varphi
(g_\alpha) \not\in G_{\alpha + 1}$. By task (II) we also have
$\varphi (g_\alpha) \not\in G_\gamma$ for any later ordinal
$\alpha < \gamma \in \lambda^*$, hence $\varphi (g_\alpha) \not\in
G$ and $\varphi$ was not an endomorphism of $G$, a contradiction.
Thus $\varphi\restriction E = r \cdot \id$ for some $r \in R$.
Thus (ii) of the theorem  holds and (iii) follows with the help of
(ii): We consider the map $\va_r= \va - r$. Since $G$ is reduced
it follows that $\va_r(K)=0$, so $\va_r$ induces a
$\tilde{\varphi}: H=G/K \to G$ which must be zero by the
hypothesis of the theorem in (iii). Thus  $\End(G)=R$ which
implies $\Hom(G,K)=0$.

Now we determine any $\varphi \in \Hom(G,H)$ and suppose $\varphi
\restriction E \ne 0$. Then by Lemma \ref{killing-two} there is
some $x \in \hE \oplus H$ such that $\varphi (x) \not\in H$.
Again, by the Black Box we can find an $\alpha \in \lambda^*$ such
that $\varphi$ extends $\varphi_\alpha$, $x \in \widehat{\dom
\varphi_\alpha}$ and $||x|| < ||\dom \varphi_\alpha||$. Hence
$\alpha$ is a bad case and $g_\alpha$ in the construction must
satisfy (I). It follows that $ \varphi_\alpha (g_\alpha) = \varphi
(g_\alpha) \not\in H$ - a contradiction. Thus $\varphi
\restriction E = 0$ and there is an induced homomorphism
$\tilde{\varphi} : H=G/K \to H$ which has to be multiplication by
some $r \in R$ by the assumptions on $H$. Thus $\varphi \in \pi R$
as required.

The mapping conditions for the cellular cover $0\arr K\arr G\arr
H\arr 0$ are obvious. Thus Theorem \ref{main2} holds. \qed

\noindent R\"udiger G\"obel, Lutz Str\"ungmann \\
Department of Mathematics,\\ University of Duisburg-Essen,\\
Campus Essen, 45117 Essen, Germany\\
{\small e-mail: ruediger.goebel@uni-due.de}\\
{\small e-mail: lutz.struengmann@uni-due.de}\\
\\
Jos\'e L. Rodr\'{\i}guez\\
\'Area de Geometr\'{\i}a y Topolog\'{\i}a,\\
Facultad de Ciencias Experimentales,\\ University of Almer\'{\i}a,\\
La ca{\~n}ada de San Urbano, 04120 Almer\'{\i}a, Spain\\
{\small e-mail: jlrodri@ual.es}\\

\end{document}